\documentclass[12pt]{amsart}

\usepackage{amsmath,amssymb,amsthm}
\usepackage{graphicx}
\usepackage{enumerate}

\usepackage[ps2pdf]{hyperref}

\newcommand{\calA}{\mathcal{A}}

\newcommand{\calC}{\mathcal{C}}

\newcommand{\calF}{\mathcal{F}}
\newcommand{\calG}{\mathcal{G}}

\newcommand{\calP}{\mathcal{P}}
\newcommand{\calQ}{\mathcal{Q}}
\newcommand{\calR}{\mathcal{R}}

\newcommand{\calT}{\mathcal{T}}

\newcommand{\calY}{\mathcal{Y}}
\newcommand{\calZ}{\mathcal{Z}}

\newcommand{\HH}{\mathbb{H}}

\newcommand{\RR}{\mathbb{R}}

\newcommand{\ZZ}{\mathbb{Z}}

\newcommand{\gothic}{\mathfrak}

\newcommand{\gL}{{\gothic L}}
\newcommand{\gR}{{\gothic R}}

\newtheorem{theorem}{Theorem}[section]
\newtheorem{proposition}[theorem]{Proposition}
\newtheorem{corollary}[theorem]{Corollary}
\newtheorem{lemma}[theorem]{Lemma}

\newtheorem{Thm}{Theorem}

\theoremstyle{definition}
\newtheorem{definition}[theorem]{Definition}

\newtheorem*{claim*}{Claim}

\newtheorem*{question*}{Question}
\newtheorem*{answer*}{Answer}
\newtheorem*{application*}{Application}

\theoremstyle{remark}
\newtheorem{remark}[theorem]{Remark}
\newtheorem*{remark*}{Remark}

\newcommand{\secref}[1]{Section~\ref{Sec:#1}}
\newcommand{\thmref}[1]{Theorem~\ref{Thm:#1}}
\newcommand{\corref}[1]{Corollary~\ref{Cor:#1}}
\newcommand{\lemref}[1]{Lemma~\ref{Lem:#1}}
\newcommand{\propref}[1]{Proposition~\ref{Prop:#1}}

\newcommand{\defref}[1]{Definition~\ref{Def:#1}}

\newcommand{\eqnref}[1]{Equation~\eqref{Eq:#1}}

\DeclareMathOperator{\size}{size}
\DeclareMathOperator{\twist}{twist}

\DeclareMathOperator{\Mod}{Mod}
\DeclareMathOperator{\Ext}{Ext}
\DeclareMathOperator{\area}{area}
\DeclareMathOperator{\diam}{diam}

\DeclareMathOperator{\I}{i}

\newcommand{\emul}{\stackrel{{}_\ast}{\asymp}}
\newcommand{\gmul}{\stackrel{{}_\ast}{\succ}}
\newcommand{\lmul}{\stackrel{{}_\ast}{\prec}}
\newcommand{\eadd}{\stackrel{{}_+}{\asymp}}
\newcommand{\gadd}{\stackrel{{}_+}{\succ}}
\newcommand{\ladd}{\stackrel{{}_+}{\prec}}

\newcommand{\SL}{\operatorname{SL}} 

\newcommand{\AC}{\mathcal{AC}} 

\newcommand{\EL}{\mathcal{EL}} 
\newcommand{\PML}{\mathcal{PML}}  

\newcommand{\Teich}{{Teichm\"uller }} 

\renewcommand{\d}{{\sf d}}

\newcommand{\param}{{\mathchoice{\mkern1mu\mbox{\raise2.2pt\hbox{$
\centerdot$}}
\mkern1mu}{\mkern1mu\mbox{\raise2.2pt\hbox{$\centerdot$}}\mkern1mu}{
\mkern1.5mu\centerdot\mkern1.5mu}{\mkern1.5mu\centerdot\mkern1.5mu}}}

\renewcommand{\setminus}{{\smallsetminus}}

\newcommand{\from}{\colon\thinspace} 
\newcommand{\ep}{\epsilon}
\newcommand{\bdy}{\partial} 

\newcommand{\sD}{{\sf D}} 
\newcommand{\sE}{{\sf E}} 
\newcommand{\sF}{{\sf F}} 
\newcommand{\sG}{{\sf G}} 
\newcommand{\sY}{{\sf Y}}

\newcommand{\bG}{\overline \calG}
\newcommand{\bL}{\overline L}
\newcommand{\bt}{\overline t}
\newcommand{\bT}{\overline T}
\newcommand{\bq}{{\overline q}}
\newcommand{\bx}{{\overline x}}
\newcommand{\by}{{\overline y}}
\newcommand{\bmu}{{\overline \mu}}
\newcommand{\blambda}{{\overline \lambda}}

\newcommand{\bI}{{\overline I}}
\newcommand{\btwist}{{\overline \twist}}

\newcommand{\rY}{ \hspace{1pt} 
\rule[-2pt]{.5pt}{8pt}_{\hspace{1pt} Y}  }
\newcommand{\stroke}[1]{ \hspace{1pt} 
\rule[-2pt]{.5pt}{8pt}_{\hspace{1pt} #1}  }

\begin{document}

\title       {Hyperbolicity in Teichm\"uller space}
\author   {Kasra Rafi}
\address {Department of Mathematics\\
               University of Oklahoma\\
               Norman, OK 73019-0315, USA}
\email      {rafi@math.ok.edu}
\maketitle
  
\begin{abstract}
We review and organize some results describing the behavior of a \Teich
geodesic and draw several applications: 1) We show that \Teich geodesics 
do not back track. 2) We show that a \Teich geodesic segment whose endpoints 
are in the thick part has the fellow travelling property. This fails when the endpoints 
are not necessarily in the thick part. 3) We show that if an edge of a \Teich
geodesic triangle passes through the thick part, then it is close to one of the other
edges.
 \end{abstract}

\section{Introduction} \label{Sec:Intro}

Two points in \Teich space determine a unique \Teich geodesic that connects
them. One would like to understand  the behavior of this geodesic and how 
the given data, two end points $x,y$ in $\calT(S)$ \Teich of a surface $S$, 
translate to concrete information about the geodesic segment $[x,y]$ 
connecting them. Much is known about this relationship. (See 
\cite{rafi:SC, rafi:CM, rafi:LT, rafi:TT}.) The first part of the paper 
is devoted to organizing and improving some of these results 
which are scattered through several papers. Accumulation of these
results provides a complete (coarse) description of a \Teich geodesic. 
One can summarized this as follows:

\begin{Thm} \label{Thm:Description}
Let $\calG\from \RR \to \calT(S)$ be a \Teich geodesic. For every subsurface $Y$, 
there is an interval of times $I_Y$ (possibly empty) where $Y$ is 
\emph{isolated} at $\calG_t$, for $t \in I_Y$.
During this interval, the restriction of $\calG$ to $Y$ behaves like a 
geodesic in $\calT(Y)$. Outside of $I_Y$, the projection
to the curve complex of $Y$ moves by at most a bounded amount.  
\end{Thm}

In fact, we know for which subsurfaces $Y$ the interval $I_Y$ is non-empty,
and in what order these intervals appear along $\RR$. 
And applying the theorem inductively, we can describe the restriction
of the geodesic to $Y$ during $I_Y$ (\secref{Proj}). 

In the rest of the paper we consider some of the implications of
the above theorem and we examine to what extend \Teich geodesics
behave like geodesics in a hyperbolic space. It is known that the \Teich
space is not hyperbolic; Masur showed that \Teich space is not $\delta$--hyperbolic 
\cite{masur:NH} and Minsky showed that the thin part of \Teich space
has a product like structure that resembles a space with positive curvature
\cite{minsky:PR}. However, there is a strong analogy between the geometry 
of \Teich space and that of a hyperbolic space. For example, 
the isometries of \Teich space are either hyperbolic, elliptic or 
parabolic \cite{thurston:GD, bers:EP} and the geodesic fellow is exponentily
mixing \cite{masur:IE, veech:TGF}. There is also a sense that \Teich space 
is hyperbolic relative to its thin parts; Masur and Misnky showed that electrified 
\Teich space is $\delta$--hyperbolic \cite{minsky:CCI}


Each application of \thmref{Description} presented in this paper examines 
how the \Teich space equipped with the \Teich metric is similar to or different 
from a relatively hyperbolic space. Apart from their individual utility, these results 
also showcase how one can apply \thmref{Description} to answer geometric problems 
in \Teich space. 

As the first application, we show that \Teich geodesics do not \emph{backtrack}. 
This is a generalization of a theorem of Masur and Minsky \cite{minsky:CCI} 
stating that the shadow of a \Teich geodesic to the curve complex is an 
un-parametrized quasi-geodesic. We show:

\begin{Thm} \label{Thm:Shadow}
The projection of a \Teich geodesic to the complex of curves of 
any subsurface $Y$ of $S$ is an un-parametrized quasi-geodesic
in the curve complex of $Y$.
\end{Thm}

This produces a sequence of markings, analogous to a resolution of a hierarchy 
\cite{minsky:CCI}, which is obtained directly froma \Teich geodesic.  

As the second application, we examine the fellow traveling properties of 
\Teich geodesics. We show:

\begin{Thm} \label{Thm:Fellow-Travel}
Consider a \Teich geodesics segment $[x,y]$ with end points $x$ and $y$ 
in the thick part. Any other geodesic segment that starts near $x$ and 
ends near $y$ fellow travels $[x,y]$. 
\end{Thm}

In contrast to above, we can provide examples where:

\begin{Thm}  \label{Thm:Counter}
When the end points of a geodesic segment are allowed to be in the thin part, 
the above theorem does not hold. 
\end{Thm}

As our third application, we prove that geodesic triangles are slim
while they pass through the thick part of \Teich space, suggesting similarities 
between \Teich space and relatively hyperbolic groups. 

\begin{Thm}  \label{Thm:Thin}
For a geodesic triangle $\triangle(x,y,z)$ in \Teich space, if
a large segment of $[x,y]$ is in the thick part, then it is either 
close to $[x,z]$ or $[y,z]$.
\end{Thm}

\subsection*{Organization of the paper}
In \secref{Comb-Description}, we make the notion of coarsely describing
a point in \Teich space precise. This means to record enough information
so that one can estimate the length of any curve on the surface and the 
distance between two points in \Teich space. It turns out that it is sufficient 
to keep track of which curves are short as well as the length and 
the twisting parameter of the short curves. 

A \Teich geodesic is the image of a quadratic differential under the
\Teich geodesic flow. In \secref{Quadratic} we discuss
how one can translate the information given by the flat structure of a
quadratic differential to obtain the combinatorial information needed
to describe a point in $\calT(S)$. 

The precise statement for the description of a \Teich geodesic and some 
related statements are given \secref{Proj}. \thmref{Shadow} is proven in 
\secref{Backtrack}, Theorems~\ref{Thm:Fellow-Travel} and \ref{Thm:Counter} 
are proven in \secref{Fellow-Travel}, and \thmref{Thin} is proven 
in \secref{Thin}.

\subsection*{Notation}  
The notation $A \emul B$ means that the ratio $A/B$ is bounded both
above and below by constants depending on the topology of $S$ only.
When this is true we say $A$ is \emph{comparable} with $B$ or
$A$ and $B$ are comparable. The notation $A\lmul B$ means that $A/B$ 
is bounded above by a constant depending on the topology of $S$. Similarly, 
$A \eadd b$ means $|A-B|$ is uniformly bounded and
$A \ladd B$ means $(B-A)$ is uniformly bounded above in both cases
by a constant that depend only on the topology of $S$. 

\subsection*{Acknowledgements} I would like to thank Saul Schleimer
for his great help and encouragement.

\section{Combinatorial description of a point in \Teich space}
\label{Sec:Comb-Description}

In this section, we discuss the notion of a marking which provides 
a combinatorial description of a point in \Teich space 
(see \defref{Marking}). Given a description of a point 
$x$ in \Teich space we are able to estimate the extremal length of 
any curve at $x$ (\thmref{Length}). Also, given the description 
of two points $x,y \in \calT(S)$, we are able to estimate the \Teich distance 
between them (\thmref{Distance}).  We first establish terminology and the 
definitions of some basic concepts.

\subsection{\Teich metric}
Let $S$ be a compact surface of hyperbolic type possibly with boundary.
The \Teich space $\calT(S)$ is the space of all conformal structures
on $S$ up to isotopy. In this paper, we consider only the \Teich
metric on $\calT(S)$. For two points $x,y \in \calT(S)$ the \Teich
distance between them is defined to be
$$
d_\calT(x,y) = \frac 12 \log \max_f K_f, 
$$
where $f \from x \to y$ ranges over all quasi-conformal maps 
from $x$ to $y$ in the correct isotopy class and $K_f$ is 
the quasi-consofmal constant of the map $f$. (See 
\cite{gardiner:QT, hubbard:TT} for background information.)
A geodesic in this metric is called a \Teich geodesic. 

\subsection*{Arcs and curves} 
By a \emph{curve} in $S$ we mean a free isotopy class of an essential simple 
closed curve and by an \emph{arc} in $S$ we mean a proper isotopy class of 
an essential simple arc.  In both cases, 
\emph{essential} means that the given curve or arc is neither isotopic 
to a point nor it can be isotoped to $\bdy S$. The definition of an arc 
is slightly different when $S$ is an annulus. In this case, an \emph{arc} 
is an isotopy class of a simple arc connecting the two boundaries of $S$, 
relative to the endpoints of the arc.   We use $\I(\alpha, \beta)$ to denote the 
geometric intersection number between arcs or curves $\alpha$ and $\beta$
and we refer to it simply as the intersection number. 

Define the arc and curve graph $\AC(S)$ of $S$ as follows: the vertices 
are essential arcs and curves in $S$ and the edges are pairs of vertices that 
have representatives with disjoint interiors.  Giving the edges length one turns 
$\AC(S)$ into a connected metric space. The following is contained in 
\cite{minsky:CCI, minsky:CCII, klarreich:BC}

\begin{theorem}
The graph $\AC(S)$ is locally infinite, has infinite diameter and is Gromov 
hyperbolic. Furthermore, its boundary at infinity can be identified with 
$\EL(S)$, the space of ending laminations of $S$. 
\end{theorem}

Recall that, $\EL(S)$ is the space of $\emph{irrational}$ laminations
in $\PML(S)$ after forgetting the measure. An irrational lamination is one 
that has non-zero intersection number with every curve. 

\subsection*{Measuring the twist}
It is often desirable to measure the number of times a curve
$\gamma$ twists around a curve $\alpha$.  This requires us to choose 
a notion of \emph{zero twisting}.  The key example is the case 
where $S$ is an annulus with a core curve $\alpha$. Then $\AC(S)$ is 
quasi-isometric to $\ZZ$.  Choose an arc $\tau \in \AC(S)$ to serve as
the origin.  Then the \emph{twist} of $\gamma \in \AC(S)$ about 
$\alpha$ is 
$$
\twist_\alpha(\gamma, \tau) = \I(\gamma, \tau),
$$
relative to choice of origin $\tau$.

In general, if $\alpha$ is a curve in $S$ let $S^\alpha$ be the 
corresponding annular cover.  A notion of zero twisting
around $\alpha$ is given by a choice of arc 
$\tau \in \AC(S^\alpha)$.  Then, for every 
$\gamma \in \AC(S)$ intersecting $\alpha$ essentially, 
we define
$$
\twist_\alpha(\gamma, \tau) = \I(\tilde \gamma, \tau),
$$
where $\tilde \gamma$ is any essential lift of $\gamma$ 
to $S^\alpha$.  Since there may be several choices for 
$\tilde \gamma$, this notion of twisting is well defined up
to an additive error of at most one. 

A geometric structure on $S$ often naturally defines a notion
of zero twisting. For example, for a given point $x \in \calT(S)$
and a curve $\alpha$, we can define twisting around $\alpha$ in $x$ as
follows: lift $x$ to a  the conformal structure $x^\alpha$ on $S^\alpha$. 
Consider the hyperbolic metric associated to $x^\alpha$ and choose $\tau$ in 
$x^\alpha$ to be any hyperbolic geodesic perpendicular to $\alpha$.  
Now, for every curve $\gamma$ intersecting $\alpha$ non-trivially, define
$$
\twist_\alpha(\gamma, x) = \twist_\alpha(\gamma, \tau) = \I(\tilde \gamma, \tau).
$$ 
Similarly, for a quadratic differential $q$ on $S$ we can define 
$\twist_\alpha(\gamma, q)$; lift $q$ to a singular Euclidean metric $q^\alpha$ 
and choose $\tau$ to be any Euclidean perpendicular arc to $\alpha$.  
(See \secref{Quadratic} for the definition of the Euclidean metric associated 
to $q$.) 

Similarly, any  foliation, arc or curve $\lambda$
intersecting $\alpha$ essentially defines a notion of zero twisting.  
Since the intersection is essential the lift $\lambda^\alpha$ of 
$\lambda$ to $S^\alpha$ contains an essential arc which we may 
use as $\tau$. Anytime two geometric objects define notions of zero
twisting, we can talk about the relative twisting between them. 
For example, for two quadratic differentials $q_1$ and $q_2$ and a curve
$\alpha$, let $\tau_1$ be the arc in $q_1^\alpha$ that is perpendicular
to $\alpha$  and $\tau_2$ be the arc in $q_2^\alpha$ that is perpendicular
to $\alpha$. Considering both these arcs in $S^\alpha$, it makes sense
to talk about their geometric intersection number. We define:
$$
\twist_\alpha(q_1,q_2) =\I(\tau_1, \tau_2).
$$
The expression $\twist_\alpha(x_1,x_2)$ for Riemann surfaces $x_1$
and $x_2$ is defined similarly. 

\subsection*{Marking} 
Our definition of \emph{marking} differs slightly from that of 
\cite{minsky:CCII} and contains more information.

\begin{definition} \label{Def:Marking}
A marking on $S$ is a triple 
$\mu=(\calP, \{l_\alpha\}_{\alpha \in\calP}, \{\tau_\alpha\}_{\alpha \in \calP})$ 
where
\begin{itemize}
\item $\calP$ is a pants decomposition of $S$. 
\item For $\alpha \in \calP$, $l_\alpha$ is a positive real number
which we think of as the length of $\alpha$. 
\item For $\alpha \in \calP$, $\tau_\alpha$ is an arc in the annular cover
$S^\alpha$ of $S$ associated to $\alpha$, establishing a notion of
zero twisting around $\alpha$. 
\end{itemize}
\end{definition} 

For a curve $\alpha$ in $S$ and $x \in \calT(S)$, we define the extrema length
of $\alpha$ in $x$ to be
$$
\Ext_x(\alpha) = \sup_{\sigma \in [x]} \frac{\ell^2_\sigma(\alpha)}{\area(\sigma)}.
$$
Here, $\sigma$ ranges over all metric in the conformal class $x$ and
$\ell_\sigma(\alpha)$ is the infimum of the $\sigma$--length of
all representatives of the homotopy class of the curve $\alpha$. 
Using the Extremal length, we define a map from $\calT(S)$ to the space of 
markings as follows:  For any $x\in\calT(S)$, let $\calP_x$ be the pants 
decomposition with the shortest extremal length in $x$ obtained using the 
greedy algorithm. For $\alpha \in \calP_x$, let $l_\alpha = \Ext_x(\alpha)$.  
As in the discussion of zero twist above, let $\tau_\alpha$ be any geodesic 
in $S^\alpha$ that is perpendicular to $\alpha$ in $x^\alpha$. We call this the 
\emph{short marking at $x$} and denote it by $\mu_x$.

As mentioned before, we can compute the extremal length of any curve 
in $x$ from the information contained in $\mu_x$ up to a multiplicative error. 
It follows from \cite{minsky:PR}  that:

\begin{theorem} 
 \label{Thm:Length-Formula}
For every curve $\gamma$, we have
$$
\Ext_x(\gamma) \emul \sum_{\alpha \in \calP} 
\left(\frac1{l_\alpha}
   + l_\alpha \cdot \twist_\alpha(\gamma, \tau_\alpha)^2 \right) \I(\alpha, \gamma)^2. 
$$
\end{theorem}

\subsection*{Subsurface Projection}
To compute the distance between two points $x,y \in \calT(S)$ we need 
to introduce the concept of subsurface projection.  We call a collection 
of vertices in $\AC(S)$ having disjoint representatives a \emph{multicurve}. 
For every proper subsurface $Y \subset S$ and any multicurve
$\alpha$ in $\AC(S)$ we can project $\alpha$ to $Y$ to obtain
a multicurve in $\AC(Y)$ as follows: let $S^Y$ be the cover of $S$ corresponding 
to $\pi_1(Y) < \pi_1(S)$ and identify the Gromov compactification of $S^Y$ with $Y$.
(To define the Gromov compactification, one needs first to pick a metric on
$S$. However, the resulting compactification is independent of the metric. 
Since $S$ admits a hyperbolic metric, every essential curve in $S$ lifts to an 
arc which has a well defined end points in the Gromov boundary of $S^Y$.)
Then for $\alpha \in \AC(S)$, the projection $\alpha \rY$ is 
defined to be the set of lifts of $\alpha$ to $S^Y$ that are essential curves or 
arcs. Note that $\alpha \rY$ is a set of diameter one in $\AC(Y)$ since all the 
lifts have disjoint interiors. 

For markings $\mu$ and $\nu$, define 
$$
d_Y(\mu,\nu)= \diam_{\AC(Y)}(\calP \rY \cup \calR \rY)
$$
where $\calP$ and $\calR$ are the pants decompositions for $\mu$ and 
$\nu$ respectively.

\subsection*{Distance Formula} 
In what comes below, the function $[a]_C$ is equal to $a$ if $a\geq C$ and it 
is zero otherwise. Also, we modify the $\log(a)$ function to be one for 
$a\leq e$. We can now state the distance formula:

\begin{theorem}[Theorem 6.1, \cite{rafi:CM}] \label{Thm:Distance}
There is a constant $C>0$ so that the following holds. 
For $x,y \in \calT(S)$ let $\mu_x= (\calP, \{l_\alpha\}, \{ \tau_\alpha\})$ and 
$\mu_y=  (\calR, \{k_\beta\}, \{ \sigma_\beta\})$ be the
associated short markings. 
Then,
\begin{align} 
d_\calT(x,y)  \asymp 
 &\sum_Y \Big[ d_Y(\mu_x, \mu_y )\Big]_C + 
        \sum_{\gamma \not \in \calP \cup \calR}
        \Big[ \log d_\gamma (\mu_x, \mu_y )\Big]_C \notag \\
 & +\sum_{\alpha \in \calP \setminus \calR} \log \frac 1{l_\alpha} +
        \sum_{\beta \in \calR \setminus \calP} \log \frac 1{k_\beta} \label{Eq:Distance} \\
  & + \sum_{\gamma \in \calP \cap \calR} 
         d_\HH \Big( \big(1/l_\gamma, \twist_\gamma(x,y) \big),
         \big(1/k_\gamma, 0\big)\Big). \notag
\end{align}
Here, $d_\HH$ is the distance in the hyperbolic plane. 
\end{theorem}

\begin{remark}
In above theorem, $C$ can be taken to be as large an needed. However,
increasing $C$ will increase the constants hidden inside $\asymp$. 
Let $\gL$ be the left hand side of \eqnref{Distance} and
$\gR$ be the right hand side. Then, a stronger version of this theorem 
can be stated as follows: There is $C_0>0$, depending only on the
topology of $S$, and for every $C\geq C_0$ there are constants $A$ and $B$ 
so that 
$$
\frac  \gL A -B \leq \gR \leq A \, \gL + B. 
$$
\end{remark}

As a corollary, we have the following criterion for showing two
points in \Teich space are a bounded distance apart. 
Let $\ep_0>\ep_1>0$, let $\calA_x$ be a set of curves in $x$ that have 
extremal length less than $\ep_0$ and assume that every other curve in 
$x$ has a length larger than $\ep_1$. Let $\ep'_0, \ep'_1$ and $\calA_y$
be similarly defined for $y$. 

\begin{corollary} \label{Cor:Bounded-Distance} 
Assume, for $x,y \in \calT(S)$, that
\begin{enumerate}
\item $\calA_x = \calA_y$ 
\item For any subsurface $Y$ that is not an annulus with core curve in $\calA_x$,  
$d_Y(\mu_x, \mu_y)=O(1)$.
\item For $\alpha \in \calA_x$, $\ell_x(\alpha) \emul \ell_y(\alpha)$. 
\item For $\alpha \in \calA_x$, 
$\displaystyle \twist_\alpha(x,y) = O\left( 1/ {\Ext_x(\alpha)}\right)$.
\end{enumerate}
Then, $d_\calT(x,y)=O(1)$. 
\end{corollary}
\begin{proof}
Condition $(2)$ implies that the first two terms in Equation \eqref{Eq:Distance}
are zero. Since $\calA_x=\calA_y$, curves in $\calP\setminus \calR$ and 
$\calR \setminus \calP$ have lengths that are bounded below. Hence the third 
and the forth terms of \eqnref{Distance} are uniformly bounded. The conditions 
on the lengths and twisting of curves in $\calA_x$ imply that the last term is 
uniformly bounded; for points $p, q \in \HH$, $p=(p_1,p_2)$, $q=(q_1,q_2)$, if 
\begin{equation*}
 (p_1 -q_1) \emul p_2 \emul q_2
 \qquad\text{then}\qquad d_\HH(p,q) =O(1). \qedhere
 \end{equation*}
\end{proof}

\section{Geometry of quadratic differentials}
\label{Sec:Quadratic}

A geodesic in \Teich space is the image of a quadratic differential
under the \Teich geodesic flow. Quadratic differentials are naturally
equipped with a singular Euclidean structure. We, however, often
need to compute the extremal length of a curve. In this section, 
we review how the extremal length of a curve can be computed from 
the information provided by the flat structure and how the flat length 
and the twisting information around a curve change along a \Teich geodesic. 

\subsection*{Quadratic differentials}
Let $\calT(S)$ be the \Teich space of $S$ and $\calQ(S)$ be the space of unit 
area quadratic differentials on $S$. Recall that a quadratic differential $q$ on 
a Riemann surface $x$ can locally be represented as 
$$
q=q(z) \, dz^2, 
$$
where $q(z)$ is a meromorphic function on $x$ with all poles having 
a degree of at most one. All poles are required to occur at the punctures. 
In fact, away from zeros and poles, there is a change of coordinates
so that $q=dz^2$. Here $|q|$ locally defines a Euclidean metric
on $x$ and the expressions $\Im(\sqrt{q})=0$ and  $\Re(\sqrt{q})=0$ define the
horizontal and the vertical directions. Vertical trajectories foliate the surface
except at the zeros and the poles. This foliation equipped with the transverse
measure $|dx|$ is called the vertical foliation and is denoted by $\lambda_-$. 
The horizontal foliation is similarly defined and is denoted by $\lambda_+$.  

A neighborhood of a zero of order $k$ has the structure of the Euclidean cone 
with total angle $(k+2)\pi$ and a neighborhood of a degree one pole has the 
structure of the Euclidean cone with total angle $\pi$.  In fact, this locally Euclidean 
structure and this choice of the vertical foliation completely determines $q$. 
We refer to this metric as the $q$--metric on $S$. 


\subsection*{Size of a subsurface}
For every curve $\alpha$, the geodesic representatives of $\alpha$
in the $q$--metric form a (possibly degenerate) flat cylinder $F_q(\alpha)$. 
For any proper subsurface $Y\subset S$, let $\sY=\sY_q$ be the representative 
of the homotopy class of $Y$ that has $q$--geodesic boundaries and that is 
disjoint from the interior of $F_q(\alpha)$ for every curve $\alpha \subset \bdy Y$. 
When the subsurface is an annulus with core curve $\alpha$ we think
of $\sF=F_q(\alpha)$ as its representative with geodesic boundary. 
Define $\size_q(Y)$ to be the $q$--length of the shortest essential curve in $Y$
and for a curve $\alpha$ let $\size_q(\sF)$ be the $q$--distance between 
the boundary components of $\sF$. When $Y$ is a pair of pants,
$\size_q(Y)$ is defined to be the diameter of $\sY$. 

\subsection*{An estimate for lengths of curves}
For every curve $\alpha$ in $S$, denote the extremal length of $\alpha$
in $x\in \calT(S)$ by $\Ext_x(\alpha)$. For constants $\ep_0>\ep_1>0$, the
$(\ep_0, \ep_1)$--thick-thin decomposition of $x$ is the pair $(\calA, \calY)$, 
where $\calA$ is the set of curves $\alpha$ in $x$ so that 
$\Ext_x(\alpha) \leq \ep_0$ and $\calY$ is the set of homotopy class of 
the components of $x$ cut along $\calA$. We further assume that 
the extremal length of any essential curve $\gamma$ that is disjoint from 
$\calA$ is larger than $\ep_1$. 

Consider the quadratic differential $(x,q)$ and the thick-thin
decomposition $(\calA, \calY)$ of $x$. Let $\alpha \in \calA$
be the common boundary of subsurfaces $Y$ and $Z$ in $\calY$. 

Let $\alpha^*$ be the geodesic representative of $\alpha$ in the boundary 
of $\sY$ and let $\sE=E_q(\alpha, Y)$ be the largest regular neighborhood of 
$\alpha^*$ in the direction of $\sY$ that is still an embedded annulus.
We call this annulus the expanding annulus with core curve $\alpha$
in the direction of $Y$. Define $M_q(\alpha, Y)$ to be $\Mod_x(\sE)$, 
where $\Mod_x(\param)$ is the modulus of an annulus in $x$. 
Recall from \cite[Lemma 3.6]{rafi:SC} that
$$
\Mod_x (\sE) \emul \log \frac{\size_q(Y)}{\ell_q(\alpha)}
\quad\text{and}\quad
\Mod_x (\sF) = \frac{\size_q(\sF)}{\ell_q(\alpha)}.
$$
Let $\sG=E_q(\alpha,Z)$ and $M_q(\alpha, Z)$ be defined 
similarly. 

The following statement relates the information about the flat lengths of 
curves to their extremal length. For a more general statement see 
\cite[Lemma 3 and Theorem 7]{rafi:LQC}.

\begin{theorem} \label{Thm:Length}
Let $(x,q)$ be a quadratic differential and let $(\calY, \calA)$ be the thick-thin 
decomposition of $x$. Then 
\begin{enumerate}
\item For $Y \in \calY$ and a curve $\gamma$ in $Y$
$$
\Ext_x(\gamma) \emul \frac{\ell_q(\gamma)^2}{\size(Y)^2}.
$$
\item For $\alpha \in \calA$ that is the common boundary of $Y, Z \in\calY$, 
\begin{align*}
\frac 1{\Ext_x(\alpha)} &\emul \log \frac{\size_q(Y)}{\ell_q(\alpha)} +
\frac{\size_q(F_q(\alpha))}{\ell_q(\alpha)}  +\log \frac{\size_q(Z)}{\ell_q(\alpha)} \\
   & \emul \Mod_x (\sE) + \Mod_x (\sF) + \Mod_x (\sG).
\end{align*}
\end{enumerate}
\end{theorem}

\subsection*{Length and twisting along a \Teich geodesic} \label{Sec:twist}
A matrix  $A \in \SL(2,\RR)$ acts on any $q \in \calQ(S)$
locally by affine transformations. The total angle at a point
does not change under this transformation. Thus the resulting singular
Euclidean structure defines a quadratic differential that we denote by $Aq$. 
The \Teich geodesic flow, $g_t \from \calQ \to \calQ$, is the action
by the diagonal subgroup of $\SL(2,\RR)$:
$$
g_t(q)= \begin{bmatrix} e^t & 0 \\ 0&e^{-t}\end{bmatrix} q.
$$
The \Teich geodesic described by $q$ is then a map
$$
\calG \from \RR \to \calQ, \qquad \calG(t)= (x_t, q_t)
$$
where $q_t=g_t(q)$ and $x_t$ is the underlying Riemann surface
for $q_t$. 

The flat length of a curve along a \Teich geodesic is well behaved.
Let the horizontal length $h_t(\alpha)$ of $\alpha$ in $q$ be the transverse 
measure of $\alpha$ with respect to the vertical foliation of $q_t$ and 
the vertical length $v_t(\alpha)$ of $\alpha$ be the transverse
measure with respect to the horizontal foliation of $q_t$. We have
(see the discussion on \cite[Page 186]{rafi:SC})
$$
\ell_{q_t}(\alpha) \emul h_t(\alpha) + v_t(\alpha).
$$
Since the vertical length decreases exponentially fast and the
horizontal length increases exponentially fast, for every
curve $\alpha$, there are constants $L_\alpha$ and $t_\alpha$ so that 
\begin{equation} \label{Eq:Flat-Length}
\ell_{q_t}(\alpha) \emul L_\alpha \cosh (t -t _\alpha). 
\end{equation}
We call the time $t_\alpha$ the balanced time for $\alpha$ and the length 
$L_\alpha$ the minimum flat length for $\alpha$. 

We define the twisting parameter of a curve along a \Teich geodesic
to be the relative twisting of $q_t$ with respect to the vertical foliation. 
That is, for any curve $\alpha$ and time $t$, let $\tau_t$ be the
arc in the annular cover of $q_t^\alpha$ that is perpendicular
to $\alpha$ and let $\lambda_-$ be the vertical foliation of $q_t$
(which is topologically the same foliation for every value of $t$). 
Define
$$
\twist_t(\alpha) = \twist_\alpha(\tau_t, \lambda_-).
$$
This is an increasing function that ranges from a minimum of zero to
a maximum of $T_\alpha=d_\alpha(\lambda_-, \lambda_+)$. That is, 
$\tau_t$ looks like $\lambda_-$ at the beginning and like $\lambda_+$
in the end. In fact, from \cite[Equation 16]{rafi:CM} we have
the following explicit formula:
\begin{equation} \label{Eq:Twist}
\twist_t(\alpha) \eadd 
  \frac{2\,T_\alpha \, e^{2(t-t_\alpha)}}{\cosh^2(t-t_\alpha)}.
\end{equation}
Also,  \cite[Proposition 5.8]{rafi:LT} gives the following estimate on the modulus
of $\sF_t=F_{q_t}(\alpha)$:
\begin{equation} \label{Eq:Modulus}
\Mod_{q_t}(\sF_t) \emul \frac{T_\alpha}{\cosh^2(t-t_\alpha)}.
\end{equation}
That is, the modulus of $\sF_t$ is maximum when $\alpha$ is balanced and 
goes to zero as $t$ goes to $\pm \infty$. The maximum modulus of $\sF_t$ is 
determined purely by the topological information $T_\alpha$, which is the relative 
twisting of $\lambda_-$ and $\lambda_+$ around $\alpha$. The size of $\sF_t$  
at $q_t$ is equal to its modulus times the flat length of $\alpha$ at $q_t$. Hence,
\begin{equation} \label{Eq:Size}
\size_{q_t}(\sF_t) =\frac {T_\alpha L_\alpha}{\cosh(t-t_\alpha)}
\end{equation}

\section{Projection of a quadratic differential to a subsurface}
\label{Sec:Restriction}

In this section, we introduce the notion of an isolated surface in a quadratic
differential. Let $(x,q)$ be a quadratic differential, $Y \subset S$ 
be a proper subsurface and $\sY$ be the representative of $Y$
with $q$--geodesic boundaries. Note that, when $\sY$ is non-degenerate, 
it is itself a Riemann surface that inherits its conformal structure from
$x$. In this case, for a curve $\gamma$ in $Y$, we use the
expression $\Ext_\sY(\gamma)$ to denote the extremal length of
$\gamma$ in the Riemann surface $\sY$. The following lemma which is a 
consequence of (\cite[Lemma~4.2]{minsky:PR}).

\begin{lemma}[Minsky] 
\label{Lem:Length-Comparison}
There exists a constant $m_0$ depending only on the topological type of $S$
so that, for every subsurface $Y$ with negative Euler characteristic the 
following holds. If $M_q(\alpha, Y)\geq m_0$ for every boundary component 
$\alpha$ of $Y$ then for any essential curve $\gamma$ in $Y$
$$
\Ext_{\sY}(\gamma) \emul \Ext_x(\gamma).
$$
\end{lemma}

Fixing $m_0$ as above, we say $Y$ is \emph{isolated} 
in $q$ if, for every boundary component $\alpha$ of $Y$, 
$M_q(\alpha,Y) \geq m_0$. The large expanding annuli in the boundaries
of $Y$ isolate it in the sense that one does not need any information about 
the rest of the surface to compute extremal lengths of curves in $Y$. 
As we shall see, when $Y$ is isolated, the restrictions of the hyperbolic
metric of $x$ to $Y$ and the quadratic differential $q$ to $Y$ are at most
a bounded distance apart in the \Teich space of $Y$.

For $x \in \calT(S)$ and $Y \subset S$ we define the Fenchel-Nielsen 
projection of $x$ to $Y$, a complete hyperbolic metric $x \rY$ on $Y$, 
as follows:  Extend the boundary curves of $Y$ to a pants decomposition 
$\calP$ of $S$.  Then the Fenchel-Nielsen coordinates of $\calP \rY$ 
defines a point $x \rY$ of $\calT(Y)$ (see \cite{minsky:PR} for
a detailed discussion).

Now, we construct a projection map from $q$ to $q\rY$ by considering
the representative with geodesic boundary $\sY$ and capping
off the boundaries with punctured disks. It turns out that the underlying
conformal structure of $q \rY$ and $x \rY$ are not very different, but
the quadratic differential restriction commutes with the action of
$\SL(2,\RR)$. When $Y$ is not isolated in $q$, the capping off process is 
not geometrically meaningful (or sometimes not possible). Hence, the process 
is restricted to the appropriate subset of $\calQ$. 

\begin{theorem} \label{Thm:Restriction}
Let $Y$ be a subsurface of $S$ that is not an annulus and let $\calQ_Y(S)$ 
be the set of quadratic differentials $q$ so that $Y$ is isolated in $q$. 
There is a map $\pi_Y \from \calQ_Y(S) \to \calQ(Y)$, with 
$\pi_Y(q) = q \rY$, so that 
\begin{equation} \label{Eq:Y}
d_{\calT(Y)}(q \rY, x \rY)=O(1).
\end{equation}
Furthermore, if, for $A \in \SL(2,\RR)$, both $q$ and $Aq$ are in $\calQ_Y (S)$ 
then
\begin{equation} \label{Eq:A}
d_{\calT(Y)} \big( (Aq) \rY, A (q \rY) \big)=O(1).
\end{equation}
\end{theorem}

\begin{proof}
We first define the map $\pi_Y$. Let $(x,q)$ be a quadratic differential 
with $Y$ isolated in $q$. Let $\sY$ be the representative of $Y$ with
$q$--geodesic boundaries. Our plan, nearly identical to that 
of~\cite{rafi:TT}, is to fill all components of $\bdy \sY$ with locally flat 
once-punctured disks. 

Fix $\alpha \subset \bdy Y$ and recall that $\sE = E_q(\alpha, Y)$
is an embedded annulus and $\alpha^*$ is a boundary of $\sE$. 
Let $a_1, \ldots, a_n$ be the points on $\alpha^*$ 
which have angle $\theta_i > \pi$ in $\sE$. Note that this set is nonempty: 
if it is empty then $\sE$ meets the interior of the flat cylinder $F(\alpha)$, 
a contradiction.  Let $\sE'$ be the double cover of $\sE$ and let 
$\alpha'$ be the pre-image of $\alpha^*$.  Let $q'$ be the lift of 
$q \stroke{\sE}$ to 
$\sE'$.  Along $\alpha'$ we attach a locally flat disk $\sD'$ 
with a well defined notion of a vertical direction, as follows.

Label the lifts of $a_i$ to $\sE'$ by $b_i$ and $c_i$. We will fill 
$\alpha'$ by symmetrically adding $2\,(n-1)$ Euclidean triangles 
to obtain a flat disk $\sD'$ such that the total angle at each $b_i$ 
and $c_i$ is a multiple of $\pi$ and is at least $2\pi$. 

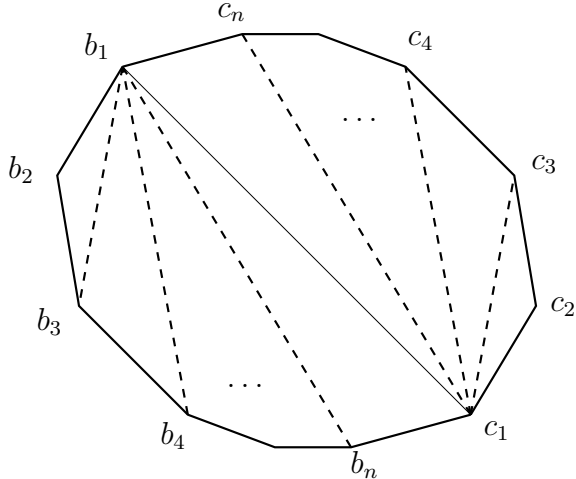
\begin{figure}[ht]
\setlength{\unitlength}{0.01\linewidth}
\begin{picture}(50,50)
\put(0,0){\includegraphics[width=65\unitlength]{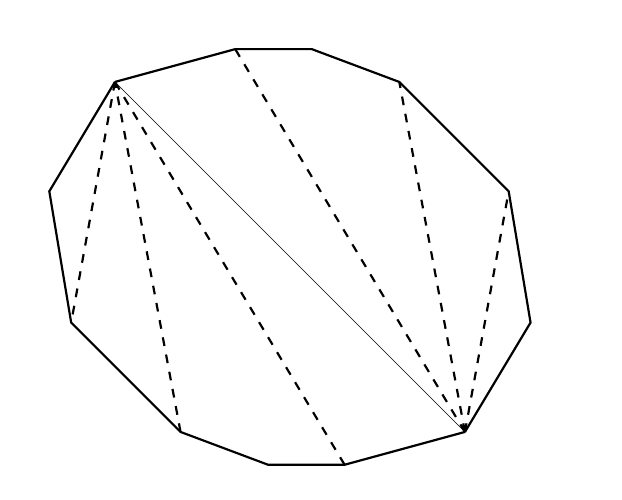}}
   \put(8,45){$b_1$}
   \put(0,32){$b_2$}  
   \put(3,16){$b_3$}   
   \put(16,4){$b_4$}   
   \put(23,10){$\ldots$} 
   \put(36,1){$b_n$}
   \put(50,5){$c_1$} 
   \put(57,18){$c_2$} 
   \put(55,33){$c_3$} 
   \put(42,46){$c_4$} 
   \put(35,38){$\ldots$} 
   \put(22,49){$c_n$} 
\end{picture}
\caption{\bf The filling of the annulus $\sE'$}
\label{fig:triangles}
\end{figure}


We start by attaching a Euclidean triangle to vertices $b_1, b_2, b_3$, 
which we denote by $\bigtriangleup(b_1,b_2,b_3)$ (see Figure~\ref{fig:triangles}). 
We choose the angle $\angle b_2$ at the vertex $b_2$ so that 
$\theta_2+\angle b_2$ is a multiple of $\pi$. Assuming
$0 \leq \angle b_2 < \pi$, there is a unique such triangle. Attach an isometric 
triangle to $c_1, c_2, c_3$. Now consider the points $b_1, b_3, b_4$.
Again, there exists a Euclidean triangle with one edge equal to
the newly introduced segment $[b_1,b_3]$, another edge equal to the 
segment $[b_3,b_4]$ and an angle at $b_3$ that makes the total angle
at $b_3$, including the contribution from the triangle 
$\bigtriangleup(b_1,b_2,b_3$), a multiple of $\pi$. 
Attach this triangle to the vertices $b_1, b_3, b_4$ and an identical triangle 
to the vertices $c_1,c_3,c_4$. Continue in this fashion until finally adding
triangles $\bigtriangleup(b_1,b_n, c_1)$ and $\bigtriangleup(c_1,c_n, b_1)$.
Due to the symmetry, the two edges connecting $b_1$ and $c_1$
have equal length, and we can glue them together.  We call the union of the
added triangles $\sD'$.  Notice that the involution on $\sE'$ extends to $\sD'$.  
Let $\sD = \sD(\alpha)$ be the quotient of $\sD'$, and note that $\sD$ is a 
punctured disk attached to $\alpha^*$ in the boundary of $\sE$. 

For $i \not = 1$, the total angle at $b_i$ and at $c_i$ is a multiple of $\pi$ 
and is larger than $\theta_i > \pi$; therefore, it is at least $2\pi$.  We 
have added $2\,(n-1)$ triangles. Hence, the sum of the total angles of all 
vertices is $2\sum_i \theta_i + 2\,(n-1)\pi$, which is a multiple of $2\pi$.
Therefore, the sum of the angles at $b_1$ and $c_1$ is also a multiple
of $2\pi$. But they are equal to each other, and each one is larger than $\pi$.
This implies that they are both at least $2\pi$. It follows that the 
quadratic differential $q'$ extends over $\sD'$ symmetrically 
with quotient an extension of $q$ to $\sD$.  

Thus, attaching the disk $\sD(\alpha)$ to every boundary component $\alpha^*$ 
in $\bdy \sY$ gives a point $q \rY \in \calQ(Y)$.  This completes the 
construction of the map $\pi_Y$. 

We now show that the distance in $\calT(Y)$ between $q \rY$ and 
$x \rY$ is uniformly bounded. For this, we examine the extremal lengths
of curves in two conformal structures. Since $Y$ is isolated in $q$, 
the boundaries of $Y$ are short in $x$. This implies, using \cite{minsky:PR} that, 
for any essential curve $\gamma$ in $Y$, the extremal lengths of $\gamma$ 
in $x$ and in $x \rY$ are comparable
\begin{equation} \label{eq:Ext}
\Ext_{x \rY}(\gamma) \emul \Ext_x(\gamma) 
\end{equation}
(see the proof of Theorem 6.1 in \cite[page 283, line 19]{minsky:PR}). 
We need to show that the extremal lengths of $\gamma$ in $q$ and in 
$q \rY$ are comparable as well. This obtain this after applying 
\lemref{Length-Comparison} twice. Once considering $\sY$ as 
a subset of $q$ and once as a subset of $q \rY$, \lemref{Length-Comparison}
implies:
$$
\Ext_x(\gamma) \emul \Ext_\sY(\gamma) \emul \Ext_{q \rY}(\gamma).
$$
Since the extremal lengths of curves are comparable, the distance in the 
$\calT(Y)$ between $x \rY$ and $q \rY$ is uniformly bounded above
\cite[Theorem 4]{kerckhoff:AG}. 

We note that defining the map $\pi_Y$ involved a choice of labeling
of the points $\{ a_i \}$. However, the above argument will work for 
any labeling.  In fact, for any labeling of points in a boundary component 
of $\sY$ in $q$, one can use the corresponding labeling $A(\sY)$ in 
$(Aq)$ so that $A (q \rY) = (A q) \rY$. Since all the different labeling 
result in points that are close in $\calT(Y)$ to $x \rY$, Equation~\eqref{Eq:A} 
holds independently of the choices made. This finishes the proof. 
\end{proof}

\section{Projection of a \Teich geodesic to a subsurface}
\label{Sec:Proj}

As mentioned before, a quadratic differential $q$ defines a \Teich geodesic
$\calG \from \RR \to \calQ(S)$ by taking
$$
\calG(t) = (x_t, q_t), \quad q_t = \begin{bmatrix} e^t & 0 \\ 0&e^{-t}\end{bmatrix} q
$$
where $x_t$ is the underlying Riemann surface for $q_t$.  Let $\lambda_+$ 
and $\lambda_-$ be the horizontal and the vertical foliations of $q_t$. 

Recall that a point $x \in \calT(S)$ has an associated shortest marking 
$\mu_x$.  We similarly define, for any $(x,q) \in \calQ(S)$, a shortest 
marking $\mu_q$.  The marking $\mu_q$ has the same pants decomposition 
and the same set of lengths $\{l_\alpha\}$ as $\mu_x$.  However, we
use the flat metric of $q$ to define the transversals $\tau_\alpha$, as follows. 
Recall that $q^\alpha$ is the annular cover of $q$ with respect to $\alpha$. 
Define $\tau_\alpha$ to be any arc connecting the boundaries of $q^\alpha$
that is perpendicular to the geodesic representative of the core. 
That is, the transversal is the quadratic differential perpendicular
instead of the hyperbolic perpendicular. 

In what follows, we often replace $q_t$ subscripts simply with $t$.
For example, $\ell_t(\alpha)$ is short for $\ell_{q_t}(\alpha)$, while
$\mu_t$ is short for $\mu_{q_t}$ and
$M_t(\alpha, Y)$ is short for $M_{q_t}(\alpha, Y)$. 
We let $t_\alpha$ be the time when $\alpha$ is
\emph{balanced} along $\calG$ (see \eqnref{Flat-Length}).
We need the following two statements. First we have a lemma that is 
contained in the proof of Theorem~3.1 of~\cite{rafi:CM}.

\begin{lemma}
\label{Lem:Convex}
There is a uniform constant $c \geq 0$ so that 
$$M_s(\alpha, Y) \leq M_t(\alpha, Y) + c$$ for all $s \leq t \leq
t_\alpha$ and for all $t_\alpha \leq t \leq s$. \qed
\end{lemma}

Second we have a theorem that follows from the proof of 
Theorem~5.5 of~\cite{rafi:SC}.

\begin{theorem}
\label{Thm:LackOfMotion}
There are constants $M_0$ and $C$ so that, if  $M_t(\alpha, Y) \leq M_0+ c$ 
for some boundary component $\alpha$, then either
\begin{equation*}
d_Y(\mu_t, \lambda_-)\leq C \quad\text{or}\quad d_Y(\mu_t, \lambda_+)\leq C. 
\end{equation*}
\end{theorem}

We now define $I_Y$, the \emph{interval of isolation} for $Y$.
Choose a large enough $M_0$ (we need $M_0 >m_0$ as in 
\lemref{Length-Comparison} and we need $M_0$ to satisfy \thmref{LackOfMotion}).
Define the interval $I_{\alpha, Y} \subset \RR$ to be empty when 
$M_{t_\alpha}(\alpha,Y) < M_0$ and otherwise to be
the largest interval containing $t_\alpha$ so that $M_t(\alpha, Y) \geq
M_0$ for all $t \in I_{\alpha, Y}$. Define
$$
I_Y = \bigcap_{\alpha \subset \bdy Y} I_{\alpha, Y}.
$$
 Note that, by 
Lemma~\ref{Lem:Convex}, for any $t$ outside of $I_Y$, there is 
a boundary component $\alpha$ such that $M_t(\alpha, Y) \leq M_0 + c$.  

\begin{theorem} \label{Thm:Proj}
Let $\calG \from \RR \to \calQ(S)$ be a \Teich geodesic with $\calG(t) = (x_t, q_t)$.
Let $Y$ be a subsurface with the interval of isolation $I_Y$. 
Then there exists a geodesic $\calF \from I_Y \to \calQ(Y)$ with 
$\calF(t)= (y_t, p_t)$, so that
\begin{itemize}
\item If $[a,b] \cap I_Y = \emptyset$ then 
$$
d_Y(\mu_a, \mu_b) = O(1).
 $$
\item For $t \in I_Y$,
$$ 
d_{\calT(Y)} \big( x_t \rY , y_t  \big) = O(1).
$$
\end{itemize}
In fact, we may take $p_t = q_t \rY$. 
\end{theorem}

\begin{figure}[ht]
\setlength{\unitlength}{0.01\linewidth}
\begin{picture}(40,37)
\put(0,0){\includegraphics[width=40\unitlength]{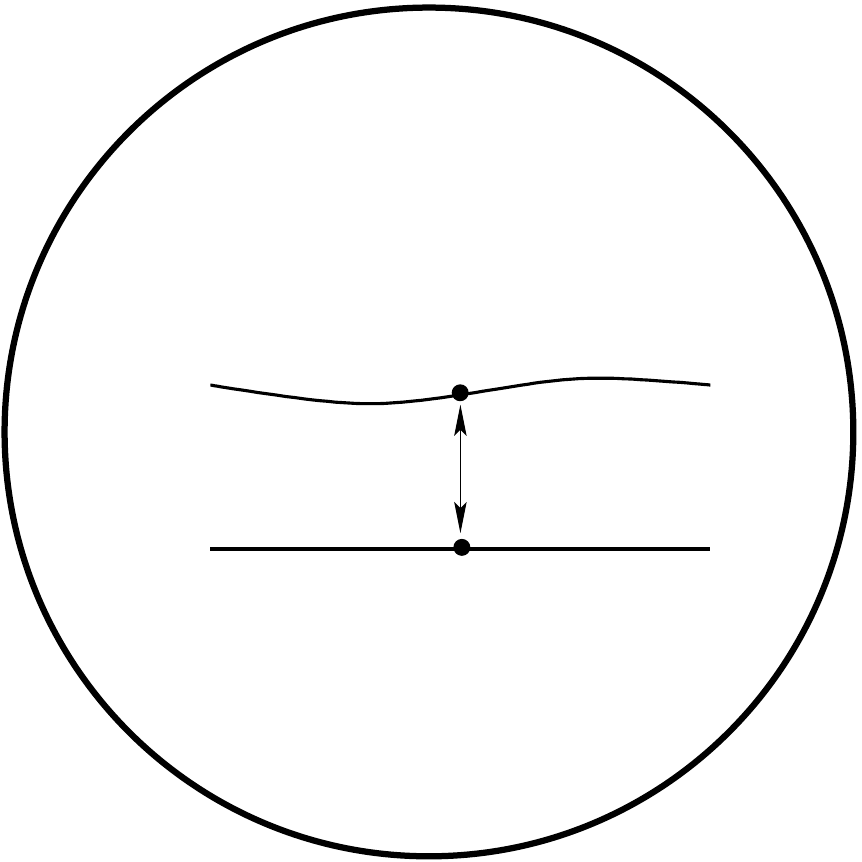}}
   \put(38,32){$\calT(Y)$}
   \put(20,23.5){$x_t \rY$}
   \put(20,11.3){$y_t$}
   \put(23,17){$O(1)$}
\end{picture}
\caption{The projection of $\calG$ to $\calT(Y)$ fellow travels the geodesic 
$\calF$.} 
\label{fig:FN} 
\end{figure}

\begin{proof}
For every $t\in[a,b]$, there exists a boundary component $\alpha$
so that $M_t(\alpha, Y) \leq M_0+c$. By \thmref{LackOfMotion}
$$
d_Y(\mu_t, \lambda_-)\leq C \quad\text{or}\quad  d_Y(\mu_t, \lambda_+)\leq C.
$$
Let $J_- \subset [a,b]$ be the set of times where former holds and
$J_+\subset [a,b]$ be the set of times where latter holds. If $J_-$ or $J_+$
is empty, we are done by the triangle inequality. Otherwise, we note
that these intervals are closed and have to intersect. This implies
that $d_Y(\lambda_-,\lambda_+)\leq 2C$. Again we are done
after applying the triangle inequality; the bound on $d_Y(\mu_a, \mu_b)$ 
is at most $4C$. This proves the first conclusion of \thmref{Proj}.

To obtain the second conclusion, we construct the candidate 
geodesic arc $\calF$ in $\calT(Y)$.  Let $I_Y=[c,d]$. As suggested 
in the statement of the theorem, let $p_c = q_c \rY$ and let 
$\calF=(y_t, p_t)$ be the geodesic segment from $[c,d] \to \calQ(Y)$ defined by
$$
p_t  = \begin{bmatrix} e^{t-c} & 0 \\ 0&e^{-t+c}\end{bmatrix} p_c.
$$
In fact, if we make consistent choices in the construction of $q_t \rY$
for different values of $t$, we have $p_t = q_t \rY$. Now 
Equation~\eqref{Eq:A} in \thmref{Restriction} implies
$$
d_{\calT(Y)}(x_t \rY, y_t)= O(1). 
$$
This finishes the proof. 
\end{proof}

For a \Teich geodesic segment whose end points are in the thick part
of the \Teich space, we can look at the short markings at the end points
of the segment instead of the horizontal and the vertical foliations, to determine
which subsurfaces are isolated along the geodesic segment. That is,
the end invariants can be taken to be the short markings instead of
the horizontal and the vertical foliations. 

\begin{corollary} \label{Cor:Ends}
Let $\calG \from \RR \to \calT(S)$ be a \Teich geodesic. Suppose $a<b$
are times so that $\calG(a)$ and $\calG(b)$ are in the thick part. 
Then, for every subsurface $Y$ we have 
\begin{itemize}
\item Either $I_Y \subset[a,b]$,
$$
\I(\lambda_- \rY, \mu_b \rY)=O(1)
\quad\text{and}\quad
\I(\lambda_+ \rY, \mu_a \rY)=O(1).
$$ 
In particular, 
$$
d_Y(\lambda_-, \lambda_+)\eadd d_Y(\mu_a, \mu_b).
$$
\item Or  $I_Y \cap [a,b] = \emptyset$ and
$$
d_Y(\mu_a, \mu_b)=O(1).
$$
\end{itemize}
\end{corollary}

\begin{proof}
Since the endpoints lie in the thick part of \Teich space, the times
$a$ and $b$ are not in any interval $I_Y$. That is, $I_Y$
is either contained in $[a,b]$ or it is disjoint from it. If $I_Y =[c,d]$
then all markings $\mu_t$, $t\in[-\infty, c]$ project to a bounded set
in $\AC(Y)$. In fact, from \cite[Theorem 5.5]{rafi:SC} we know that
$\I(\mu_t \rY, \lambda_+\rY)=O(1)$. Therefore, $d_Y(\lambda_+, \mu_a)=O(1)$.
Similarly, for $t \in [d,\infty]$, $\I(\mu_t \rY, \lambda_-\rY)=O(1)$
and $d_Y(\lambda_-, \mu_b)=O(1)$. The corollary follows
immediately.
\end{proof}

\subsection*{Order of appearance of intervals $I_Y$}
By examining the subsurface projections one can determine
which curves $\alpha$ are short along a \Teich geodesic $\calG$. 
The following is the restatement of results in \cite{rafi:SC} in a way
that is more suitable for our purposes. Let $\calG$ be a \Teich
geodesic with horizontal and vertical foliations $\lambda_\pm$ and,
for a curve $\alpha$, let $\calZ(\alpha, D)$ be the set of subsurface $Z$ that 
are disjoint from $\alpha$ and have $d_Z(\lambda_+, \lambda_-) \geq D$. 

\begin{theorem} \label{Thm:M-Large}
 A curve $\alpha$ is short at some point along $\calG$ 
if and only if $\alpha$ is the boundary of a subsurface $Y$ so that 
$Y$ is filled with subsurfaces with large projections. That is, there
are constants $\ep$, $D_0$ and $D_1$ so that
\begin{itemize}
\item If $\Ext_t(\alpha) \leq \ep$ then $\alpha$ is a boundary component
of some subsurface $Y$, where $Y$ is filled by subsurfaces in 
$\calZ(\alpha, D_0)$.
\item Suppose that $\alpha$ is a boundary component of $Y$ and that $Y$ is
filled by elements of $\calZ(\alpha, D_1)$.  Then there is a time 
$t \in \RR$ when $\Ext_t(\alpha) \leq \ep$.
\end{itemize}
\end{theorem}

\begin{proof}
This is a restatement of \cite[Theorem 1.1]{rafi:SC} after the following:
two curves or arcs in $\AC(Y)$ have large intersection number if and only if 
their projections to some subsurface $Z$ of $Y$ is large.
(This assertion is well known and follows from \cite[Corollary D]{rafi:TL}.)  
We have just translated the condition about intersection numbers to a condition
about subsurface projections. 
\end{proof}

One consequence of the above theorem is that the order in which 
the intervals $I_Y$ appear in $\RR$ is essentially determined by any
geodesic $g$ in $\AC(S)$ connecting $\lambda_-$ to $\lambda_+$. 

\begin{proposition} \label{Prop:Order}
The boundary curves of any isolated surface are in a $2$--neighborhood of 
a geodesic $g$ in the curve complex. The order of appearance of 
intervals of isolations in $\RR$ is coarsely determined by the order 
in which the vertices $\partial Y$ appear along $g$. 
\end{proposition}

The proof uses both the description of a \Teich geodesic as well as 
some hyperbolicity result for the curve complex $\calC(S)$. Namely,
we use Masur and Minsky's bounded geodesic image theorem:

\begin{theorem}[Theorem 3.1 in \cite{minsky:CCII}] \label{Thm:Bounded}
If $Y$ is an essential subsurface of $S$ and $g$ is a geodesic in 
$\AC(S)$ all of whose vertices intersect $Y$ nontrivially, then the projected 
image of $g$ in $\AC(Y)$ has uniformly bounded diameter. 
\end{theorem}

\begin{proof}[Proof of \propref{Order}]
By \thmref{M-Large}, a boundary curve $\alpha$ of any isolated subsurface 
$Y$ is disjoint from some subsurface $Z$ where the projection
distance $d_Z(\lambda_+, \lambda_-)$ is large. By \thmref{Bounded}, 
the geodesic $g$ has to miss $Z$ as well.
Hence $\alpha$ has a distance of at most $2$ from $g$. 

Write $g= g_- \cup g_0 \cup g_+$, where $Z$ intersects every curve
in $g_-$ and $g_+$ and where $g_0$ has length $10$ and $\partial Z$ 
is disjoint from a curve at the middle of $g_0$. From \thmref{Bounded} we 
have that the projection of $g_-$ to $\AC(Z)$ is in a bounded neighborhood of
$\lambda_- \stroke{Z}$ and the projection of $g_+$ to $\AC(Z)$ is in a
bounded neighborhood of $\lambda_+ \stroke{Z}$.  Let $Y'$ be another isolated 
surface. We claim that if the boundary of $Y'$ is close to a point in $g_-$, then
the interval $I_{Y'}$ appears after the interval $I_Y$. 

Let $I_Y =[a,b]$ and let $t \in I_{Y'}$. Then $t \not \in [a,b]$ because $Y$ and 
$Y'$ intersect (the distance between their boundaries is larger than $1$) 
and their boundaries can not be short simultaneously. Note that 
$\partial Y'$ are part of the short marking $\mu_t$. 
By \corref{Ends}, if $t<a$ then $\I(\mu_t \rY, \lambda_+\rY)=O(1)$. 
Hence,  
$$
\I(\mu_t \stroke{Z}, \lambda_+ \stroke{Z})=O(1)
\quad\text{and}\quad
d_Z(\mu_t, \lambda_+)=O(1).$$ 
But this is a contradiction because $\partial Y'$ is close to a point in $g_-$
which projects to a point in $\AC(Z)$ near $\lambda_- \stroke{Z}$. 
Therefore, $t >b$. 
\end{proof}

\begin{remark} \label{Rem:Ends}
Note that, using \corref{Ends}, we can restate the above statements
for \Teich geodesic segments $\calG \from [a,b] \to \calT(S)$ where
$\calG(a)$ and $\calG(b)$ are in the thick part. All statements hold after
replacing $\lambda_-$ and $\lambda_+$ with $\mu_a$ and $\mu_b$
respectively. 
\end{remark}

\section{No Back-tracking} \label{Sec:Backtrack}

As before, let $\calG$ be a \Teich geodesic with $\calG(t) = (x_t, q_t)$ and
let $\mu_t$ be the short marking associated to $q_t$. 
In this section we examine the projection of markings $\mu_t$
to the curve complex of a subsurface. 
 
\begin{theorem}
\label{Thm:No-Back-Tracking}
For every subsurface $Y$ of $S$, the shadow of $\calG$ in $\AC(Y)$ is an 
un-parametrized quasi-geodesic. That is, for $r\leq s \leq t \in \RR$
 $$
 d_Y(\mu_r, \mu_s)+d_Y(\mu_s, \mu_t) \ladd  d_Y(\mu_r, \mu_t).
 $$
\end{theorem}

\begin{remark}
We observe that the projection of $\mu_t$ to $\AC(Y)$ is a 
coarsely continuous path. That is, there is a constant $B$ so that
for every $t \in \RR$ there is a $\delta$ where
$$
\I(\mu_t, \mu_{t+\delta})\leq B \qquad\text{and hence}\qquad
d_Y(\mu_t, \mu_{t+\delta})=O(1).
$$
To see this, note that since lengths change continuously, $x_t$ and 
$x_{t+\delta}$ have the same thick-thin decompositions and the intersection
between moderate length curves in $x_t$ and $x_{t+\delta}$ is bounded. 
Also, twisting along the short curves changes coarsely continuously 
(see \eqnref{Twist}). 
\end{remark}

\begin{remark}
The reverse triangle inequality for a path (as given in the statement 
of the theorem) is a stronger condition than being a unparametrized 
quasi-geodesic.  However, in Gromov hyperbolic spaces 
such as $\AC(Y)$ the two conditions are equivalent. 
(See \cite[Section 7]{minsky:CCII} and \cite[Section 2.1]{masur:TS}
for relevant discussions.)
\end{remark}

\begin{remark}
This contrasts with the way geodesics behave in the Lipschitz 
metric on $\calT(S)$, studied by Thurston in \cite{thurston:MSM}, where the 
projection of a geodesic to a subsurface can backtrack 
arbitrarily far. (Examples can easily be produced using Thurston's construction 
of minimal stretch maps \cite{thurston:MSM} and the results in \cite{rafi:TL}). 
\end{remark}

\begin{proof}
If $Y=S$, the above is a theorem of Masur and 
Minsky~\cite[Theorem 3.3]{minsky:CCII}, that is, we already know that the 
shadow of $\calG$ to $\AC(S)$ is an unparametrized quasi-geodesic. 
Let $Y$ be a proper subsurface and consider the interval of isolation
$I_Y=[c,d]$. If $Y$ is not an annulus, by the first part of \thmref{Proj}, 
the shadow of $\calG(-\infty, c]$ and $\calG[d, \infty)$ have bounded diameter in 
$\AC(Y)$ and by the second part of \thmref{Proj} and again using 
\cite[Theorem 3.3]{minsky:CCII}, the shadow of $\calG[c,d]$ is an unparametrized
quasi-geodesic in $\AC(Y)$. It remains to check the case of an annulus. 
But in this case $\AC(Y)$ is quasi-isometric to $Z$ and we need only
to show that the twisting around the core of $Y$ is an increasing up to an 
additive error. This follows from \eqnref{Twist}. 
\end{proof}

\section{Fellow traveling}
\label{Sec:Fellow-Travel}

\begin{theorem} \label{Thm:FellowTravel}
There is a constant $D > 0$ so that, for points $x, \bx, y$ and $\by$  
in the thick part of $\calT(S)$ where
$$
d_\calT(x, \bx) \leq 1 \quad\text{and}\quad d_\calT(y, \by) \leq 1,
$$ 
the geodesic segments $[x,y]$ and $[\bx, \by]$ $D$--fellow 
travel in a parametrized fashion. 
\end{theorem}

\begin{remark}
The proof also works when either $x$ or $y$ is replaced with measured
foliation in $\PML(S)$ and $\calG$ and $\bG$ are infinite rays. 
\end{remark}

\begin{proof}
After adjusting $x$ and $y$ along the geodesic extension 
through $[x,y]$ by a bounded amount, we may assume that 
$d_\calT(x,y)=d_\calT(\bx,\by)$.  Let
$$
\calG \from [0,l] \to \calT(S) \quad\text{and}\quad 
\bG \from [0,l] \to \calT(S)
$$
be \Teich geodesics connecting $x$ to $y$ and $\bx$ to $\by$ respectively;
$\calG(t) = (x_t, q_t)$ and $\bG(t) = (\bx_t, \bq_t)$.

We first show that, for any curve $\alpha$, 
$\ell_{q_t}(\alpha)\emul \ell_{\bq_t}(\alpha)$. 

Since $x$ and $\bx$ are both in the thick part, for every curve $\alpha$
we have (part (1) of \thmref{Length})
$$
\Ext_x(\alpha) \emul l_{q_0}(\alpha)^2
\quad\text{and}\quad
\Ext_\bx (\alpha) \emul l_{\bq_0}(\alpha)^2.
$$
But $d_\calT(x, \bx)=1$. Therefore, 
$$
\Ext_x(\alpha) \emul \Ext_\bx(\alpha)
\quad\Longrightarrow\quad
l_{q_0}(\alpha) \emul l_{\bq_0}(\alpha).
$$ 
The same argument works to show that 
$l_{q_l}(\alpha) \emul l_{\bq_l}(\alpha)$.
The flat length of a curve is essentially determined by two parameters. 
From \eqnref{Flat-Length} we have 
$\ell_{q_t}(\alpha) \emul L_\alpha \cosh(t-t_\alpha)$ and  
$\ell_{\bq_t}(\alpha) \emul \bL_\alpha \cosh(t-\bt_\alpha)$. 
Since, the flat lengths of $\alpha$ are comparable at the beginning
and the end they are always comparable. That is, 
$L_\alpha \emul \bL_\alpha$ and $t_\alpha \eadd \bt_\alpha$. 

We use \corref{Bounded-Distance} to prove $d_\calT(x_t, \bx_t)=O(1)$
by checking the four conditions. 

\subsection*{Condition (1)} We need to show that $q_t$ and $\bq_t$ have the
same thick-thin decompositions. Fix an $\ep$ and let $(\calA, \calY)$
be the $(\ep, \ep)$--thick-thin decomposition of $x_t$. 
Let $\alpha \in \calA$ and let $\sE, \sF$ and $\sG$ be as in \thmref{Length}.
Since $\alpha$ is short, one of $\sE$, $\sF$ or $\sG$ must have a large 
modulus. That is, for every curve $\beta$ intersecting $\alpha$,
we have
$$
\frac{\ell_{q_t}(\beta)}{\ell_{q_t}(\alpha)} \gmul \frac 1\ep.
$$
(In fact it may be larger than $e^{1/e}$.) 
Since the flat length in $q_t$ and $\bq_t$ are comparable, we also have
$$
\frac{\ell_{\bq_t}(\beta)}{\ell_{\bq_t}(\alpha)} \gmul \frac 1\ep.
$$
We show the extremal length of $\alpha$ is small in $\bx_t$. If not, $\alpha$ 
would pass through some thick piece of $\bx_t$ and it would intersect some 
curve $\beta$ with $\Ext_{\bx_t}(\beta) \lmul 1$. That is, 
$\Ext_{\bx_t}(\beta) \lmul \Ext_{\bx_t}(\alpha)$. 
Part (1) of \thmref{Length} implies $\ell_{\bq_t}(\beta) \lmul \ell_{\bq_t}(\alpha)$ 
which is a contradiction.  That is, there is an $\ep_0$ so that if 
$Ext_{x_t}(\alpha) \leq \ep$ then $\Ext_{\bx_t}(\alpha) \leq \ep_0$.

Arguing in the other direction, we can find $\ep_1$ so that if 
$\Ext_{\bx_t}(\alpha) \leq \ep_1$ then $\Ext_{x_t}(\alpha) \leq \ep$. 
That is, every curve not in $\calA$ is $\ep_1$--thick in $\bx_t$. 
This proves that $(\calA, \calY)$ is a $(\ep_0, \ep_1)$--thick-thin
decomposition for $\bx_t$. 

\subsection*{Condition (2)} 
The size of a surface $Y \in \calY$ is the flat length of the shortest 
essential curve in $Y$. Hence, we have 
$\size{q_t}(Y) \emul \size_{\, \bq_t}(Y)$. 
Now, \thmref{Length} implies that, for every curve $\gamma$ in $Y$,  
if $\Ext_{x_t}(\gamma) \emul1$ then $\Ext_{\bx_t}(\gamma) \emul1$
as well. But two curves of length one have bounded intersection numbers. 
Hence, they have bounded projection to every subsurface $Z$. This means 
$d_Z(\mu, \bmu)=O(1)$.

\subsection*{Condition (3)}  For each $\alpha \in \calA$, as we saw before,
$L_\alpha \emul \overline{L}_\alpha$ and $t_\alpha \eadd \bar t_\alpha$. 
We now show that $T_\alpha \eadd \overline{T}_\alpha$. 
Since the end points of $\calG$ and $\bG$ are close, we have
\begin{align*}
&d_\alpha(\mu_0, \bmu_0)=O(1)
&&\text{and}
&& d_\alpha(\mu_l, \bmu_l)=O(1).\\
\intertext{Also, from \corref{Ends} we have}
&d_\alpha(\mu_0, \lambda_-)=O(1), \quad 
&& &&d_\alpha(\mu_l, \lambda_+)=O(1),\\
&d_\alpha(\bmu_0, \blambda_-)=O(1)
&&\text{and}
&&d_\alpha(\bmu_l, \bmu_+)=O(1).
\end{align*}
Hence, using the triangle inequality,
$$
T_\alpha=d_\alpha(\lambda_-, \lambda_+) \eadd d_\alpha(\mu_0, \mu_l) \eadd 
d_\alpha(\bmu_0, \bmu_l) \eadd d_\alpha(\blambda_-, \blambda_+)=
\overline{T}_\alpha.
$$
Now \eqnref{Modulus} implies 
\begin{equation} \label{Eq:F-Equal}
\Mod_{x_t}(\sF_t) \emul \Mod_{\bx_t}(\overline \sF_t).
\end{equation}
Also, as seen above, the size of all subsurfaces are comparable
in $q_t$ and $\bq_t$. Therefore, by \thmref{Length}
$\Ext_{x_t}(\alpha) \emul \Ext_{\bx_t}(\alpha)$.

\subsection*{Condition (4)}
We show that $\twist_\alpha(q_t, \bq_t) \Ext_{x_t}(\alpha) \emul 1$. Note that,
since $d_\alpha(\lambda_-, \blambda_-)=O(1)$, 
\begin{equation} \label{Eq:Difference}
\twist_\alpha(q_t, \bq_t) \eadd 
|\twist_\alpha(q_t, \lambda_-) - \twist_\alpha(\bq_t, \blambda_-)|.
\end{equation}
Denote $\twist_\alpha(q_t, \bq_t)$ (as before) by $\twist_t(\alpha)$
and denote $\twist_\alpha(\bq_t, \blambda_-)$ by $\btwist_t(\alpha)$.
We use \eqnref{Twist} and the facts $|t_\alpha-\bt_\alpha|=O(1)$
and $|T_\alpha -\bT_\alpha|=O(1)$ to estimate the right hand side
of \eqnref{Difference}. 

If $t \ladd t_\alpha$ (and hence $ t \ladd \bt_\alpha$), then 
$$
\twist_t(\alpha) \lmul \frac{T_\alpha}{\cosh^2(t-t_\alpha)} 
\qquad\text{and}\qquad
\btwist_t(\alpha) \lmul \frac{T_\alpha}{\cosh^2(t-t_\alpha)}.
$$
But $\Ext_t(\alpha) \lmul \frac 1 {\Mod(\sF_t)}$.   Thus using \eqnref{Modulus}
$$
\Big|\twist_t(\alpha) - \btwist_t(\alpha) \Big| \Ext_t(\alpha) 
  \lmul \frac{T_\alpha}{\cosh^2(t-t_\alpha)}  \frac{\cosh^2(t-t_\alpha)}{T_\alpha}
   \lmul 1
$$

If $t \gadd t_\alpha$, then 
$$
T_\alpha - \twist_t(\alpha) \lmul \frac{T_\alpha}{\cosh^2(t-t_\alpha)} 
\quad\text{and}\quad
T_\alpha - \btwist_t(\alpha) \lmul \frac{T_\alpha}{\cosh^2(t-t_\alpha)}.
$$
Hence, as before, 
\begin{align*}
\Big|\twist_t(\alpha) - \btwist_t(\alpha)\Big| \Ext_t(\alpha) 
 & \lmul 
 \frac{\Big|(T_\alpha -\twist_t(\alpha)) - (T_\alpha -\btwist_t(\alpha)\Big|}
    {\Mod_t(\alpha)}\\ &\\
 & \lmul  \frac{T_\alpha}{\cosh^2(t-t_\alpha)} \frac{\cosh^2(t-t_\alpha)}{T_\alpha}
   \lmul 1.
\end{align*}
That is, the last condition in \corref{Bounded-Distance} holds and
$d_\calT(q_t, \bq_t)=O(1)$. This finishes the proof. 
\end{proof}

We now construct the counterexample. 

\begin{theorem} \label{Thm:Not-FL}
For every constant $\d > 0$, there are points 
$x, y, \bx$ and $\by$ in $\calT(S)$ so that
$$
d_\calT(x, \bx) = O(1) \quad\text{and}\quad d_\calT(y, \by) =O(1),
$$
and
$$
d_\calT\big( [x,y], [\bx, \by] \big) \gmul \d.
$$
\end{theorem}

\begin{proof}
For a given $\d$, we construct quadratic differentials $q_0$ and $\bq_0$
with the following properties: Let $q_t$ be the image
of $q_0$ under the \Teich geodesic flow and let $x_t$ be the underlying 
conformal structures of $q_t$. Let $\bq_t$ and $\bx_t$ be defined
similarly. We will show that 
$$
d_\calT(x_0, \bx_0)= O(1),  \qquad 
d_\calT(x_{2\d}, \bx_{2\d})= O(1), 
$$
and
$$
d_\calT(x_\d, \bx_\d) \gmul \d. 
$$
This is sufficient to show that 
$d_\calT (x_\d,\bx_{t})\gmul \d$ for any $t \in [0,2\d]$. 
To see this note that, for any $0\leq t<\d$, we have
$$
d_\calT(x_\d, \bx_t)+ d_\calT(\bx_t, \bx_0) \gadd \d
\quad\text{and}\quad
d_\calT(x_\d, \bx_t)+ d_\calT(\bx_t, \bx_\d) \gadd d_\calT(x_\d, \bx_\d).
$$
Summing up both sides, we get
$$
2d_\calT(x_\d, \bx_t) + d_\calT(\bx_0, \bx_\d) \gadd \d + d_\calT(x_\d, \bx_\d). 
$$
Hence, 
$$
2d_\calT(x_\d, \bx_t) \gadd d_\calT(x_\d, \bx_d).
$$ 
A similar argument works for $\d < t \leq 2\d.$

Let $S$ be a surface of genus $2$, $\gamma$ be a separating curve in 
$S$ and $Y$ and $Z$ be the components of $S \setminus \gamma$. 
Consider a pseudo-Anosov map $\phi$ on a torus and choose a flat torus
$T$ on the axis of $\phi$ so that the vertical direction in $T$ matches the
unstable foliation of $\phi$. Cut open a slit in $T$ of size $\ep = c \, e^{-\d/2}$ 
and of angle $\pi/4$ (The constant $0<c<1$ is to be specified below). 
Fix a homeomorphism from $Y$ to this slit torus
and call this marked flat surface $T_0$. Define
$$
T_t= \begin{bmatrix} e^t & 0 \\ 0 & e^{-t} \end{bmatrix} T_0.
$$
Note that $T_t$ is still a marked surface.  The length of the slit
is minimum at $t=0$ and grows exponentially as $t \to \pm \infty$. 
For $-\d/2 \leq t \leq \d/2$, the length of the slit is smaller than $c$
but the length of shortest essential curve in $T_t$ in this interval is comparable 
with $1$. Hence, for $c$ small enough, $M_t(\gamma, Y) \geq m_0$
(see \secref{Restriction}) and $T_t$ looks like an isolated subsurface. 

Now choose $\delta \ll \ep$ 
(specified below) and let $q_0$ be the quadratic differential defined by 
gluing $T$ to $\delta \, T_{-\d/2}$. What we mean by this is 
that we first scale down $T_{-\d/2}$ by a factor $\delta$. Then we cut open a slit 
in $T$ of the same size and angle as the size of the slit in $\delta T_{-\d/2}$ and 
then glue these two flat tori along this slit. Fixing a homeomorphism from $Z$ to 
$T$ slit open, we obtain a marking for $q_0$ that is well defined up to 
twisting around $\gamma$. Let $\calG \from [0,2\d] \to \calT(S)$ be the
\Teich geodesic segment defined by $q_0$. 

Construct $\bq_0$ in the similar fashion by gluing $T$ to $\delta \, T_{-3\d/2}$. 
Now choose the marking map from $S$ to $\bq_0$ so that $q_0$
and $\bq_0$ have bounded relative twisting around $\gamma$. 
Let $\bG \from [0,2\d] \to \calT(S)$ be the
\Teich geodesic segment defined by $\bq_0$. 

Recall that, for $-\d/2 \leq t \leq \d/2$, the subsurface $\delta T_t$ is isolated
(scaling by $\delta$ does not change the value of $M_t(\alpha, Y)$) 
and by \thmref{Restriction} the projection of $T_t$ to the \Teich space of $Y$ 
fellow travels a \Teich geodesic. However, for $t>\d/2$ and $t<-\d/2$, the 
projection to the curve complex of $Y$ changes by at most a bounded amount.
That is, the interval of isolation for $Y$ along $\calG$, $I_Y =[0,\d]$ and
along $\bG$, $\bI_Y =[\d,2\d]$. In particular, 
$$
d_Y(q_0, \bq_0)=O(1) \quad\text{and}\quad d_Y(q_{2\d}, \bq_{2\d})=O(1).
$$ 
Also, since no curve in $Y$ or $Z$ is ever short (the vertical and the horizontal
foliation in $Y$ and $Z$ are co-bounded), the twisting parameters around any 
curves inside $Y$ or $Z$ are uniformly bounded. Projections of $q_0$ and 
$\bq_0$ to $Z$ are identical and $\gamma$ is short in both $q_0$ and $\bq_0$.
Therefore, to show $d_\calT(x_0, \bx_0)=O(1)$, it remains to show 
(\corref{Bounded-Distance}) that the extremal lengths of
$\gamma$ in $x_0$ and $\bx_0$ are comparable. We have
(\thmref{Length})
$$
\Ext_{x_0}(\gamma) \eadd \log \frac 1\delta 
\qquad\text{and}\qquad
\Ext_{\bx_0}(\gamma) \eadd \log \frac 1{e^\d \delta} =  \log \frac 1\delta - \d.
$$
But these quantities are comparable for $\delta$ small enough. A similar
argument shows that $d_\calT(x_{2\d}, \bx_{2\d})=O(1)$. 
Since $Y$ is isolated in $q_t$ for $0\leq t \leq \d$ the shadow
to the $\AC(Y)$ is an unparametrized quasi-geodesic. In fact, 
since no curve is short in $Y$ in that interval, the shadow is a parametrized
quasi-geodesic  (\cite[Lemma 4.4]{rafi:CC}). That is
$$
d_{\calT(Y)}(x_0, x_\d) \emul \d.
$$
But the interval of isolation for $Y$ along the geodesic $\bG$ is 
$[\d, 2\d]$. Therefore, 
$$
d_Y(\bx_0, \bx_\d) =O(1).
$$
As before, we have $d_Y(x_0, \bx_0)=O(1)$. Hence
$$
d_Y(q_\d, \bq_\d)\emul \d.
$$
Now, by \thmref{Distance}, we have
$$
d_\calT(x_\d, \bx_\d) \gmul d_Y(x_\d, \bx_\d) \emul \d. 
$$
This finishes the proof. 
\end{proof}

\section{Thin triangles}
\label{Sec:Thin}

Let $x$, $y$ and $z$ be three points in $\calT(S)$ and
let $\calG \from [a,b] \to \calT(S)$ be the \Teich geodesic
connecting $x$ to $y$. In this section we prove Theorem~\ref{Thm:Thin}
from the introduction. 

\begin{theorem} \label{Thm:Thin-Triangle}
For every $\ep$, there are constants $C$ and $D$ so that the following holds.
Let $[c,d]$ be a subinterval of $[a,b]$ with $(d-c)>C$ so that for every $t \in [c,d]$, 
$\calG(t)$ is in the $\ep$--thick part of $\calT(S)$.
Then, there is a $w \in [\calG(c), \calG(d)]$ where 
$$
\min \Big(d_\calT\big(w, [x,z]\big), d_\calT\big(w, [x,z]\big) \Big) \leq D.
$$
\end{theorem}

\begin{proof}
Consider the shadow map from $\calT(S)$ to the curve comlex
$\AC(S)$ sending a point $x$ to its short marking $\mu_x$. 
The geodesic triangle $\triangle(\mu_x,\mu_y,\mu_z)$ in the arc
and curve complex $\AC(S)$ is $\delta$--slim. Since the shadow of
$[x,y]$ is a quasi-geodesic (Theorem~\ref{Thm:Shadow})
for any $w \in [x,y]$, $\mu_w$ is $\delta$--close to the geodesic
$[\mu_x, \mu_y]$ in $\AC(S)$. That is, for every $w\in[x,y]$,
there is a Riemann surface $u$ in either $[x,z]$ or $[y,z]$ so that 
$d_S(\mu_w, \mu_u)\leq 3\delta$. 

The projection of $[\calG(c), \calG(d)]$ to $\AC(S)$ is in fact 
a parametrized quasi-geodesic (\cite[Lemma 4.4]{rafi:CC}).
Hence, by making $C$ large, we can assume that the shadow of 
$[\calG(c), \calG(d)]$ is as long as we like. Thus, we can choose 
$w \in [\calG(c), \calG(d)]$ so that $\mu_w$ is far from either 
the shadow of $[x,z]$ or the shadow of $[y,z]$. 
To summarize, without loss of generality, we can assume that 
there is a $w \in [\calG(c), \calG(d)]$ and a $u \in [x,z]$ so that 
$d_S(\mu_w, \mu_u)=O(1)$ and that neither $\mu_u$ nor $\mu_w$
is in the $(10\delta)$--neighborhood of the geodesic $[\mu_y,\mu_z]$.

We claim that $u$ is in the thick part of \Teich space. Using 
\thmref{M-Large} it is enough to show, for every
subsurface $Y$ whose boundaries are close to $\mu_u$ in $\AC(S)$, that
$d_Y(\mu_x, \mu_z)=O(1)$. Since $\mu_u$ is far away from 
$[\mu_y, \mu_z]$, \thmref{Bounded} implies that
$d_Y(\mu_y, \mu_z)=O(1)$. To prove the claim, we need to show that 
\begin{equation} \label{d_Y}
d_Y(\mu_x, \mu_y)=O(1).
\end{equation}

We prove \eqref{d_Y} by contradiction. Assume $d_Y(\mu_x, \mu_y)$ is large. 
By \thmref{M-Large}, $\partial Y$ is short at some point $v \in[x,y]$.
But the shadow of $[x,y]$ is a quasi-geodesic
and the shadow of $[\calG(c), \calG(d)]$ is a parametrized quasi-geodesic. 
Hence, by choosing $C$ large enough, we can conclude that, for any
such subsurface, $d_S(\partial Y, \mu_w)\gadd d_S(\mu_v, \mu_w)$ is large. 
This contradicts the fact that 
$$
d_S(\partial Y, \mu_u)=O(1)
\quad\text{and}\quad d_S(\mu_u, \mu_w)=O(1).
$$
Hence, \eqref{d_Y} holds and thus $u$ is in the thick part of \Teich space. 

We now claim, for any subsurface $Y\subset S$, that 
$$
d_Y(\mu_u, \mu_w)=O(1).
$$
This is because any such subsurface $Y$ should appear near the 
curve complex geodesic connecting $\mu_u$ and $\mu_w$ and hence
$\partial Y$ has a bounded distance from $\mu_w$ in $\AC(S)$. As before, 
assuming $d_Y(\mu_x, \mu_y)$ is large will result in a contradiction. 
Thus,  $d_Y(\mu_x, \mu_y)=O(1)$. Since $\mu_u$ is far from the geodesic 
$[\mu_y,\mu_z]$, the bounded projection theorem implies that 
$d_Y(\mu_y, \mu_z)=O(1)$ and by the triangle inequality, 
$d_Y(\mu_x, \mu_z)=O(1)$. On the other hand, by \thmref{No-Back-Tracking}
$$
d_Y(\mu_x, \mu_y)= O(1) \quad\Longrightarrow\quad d_Y(\mu_x, \mu_w)= O(1)
$$
and
$$
d_Y(\mu_x, \mu_z)= O(1) \quad\Longrightarrow\quad d_Y(\mu_x, \mu_u)= O(1).
$$
The triangle inequality implies $d_Y(\mu_w, \mu_u)= O(1)$. This proves the claim.

We have $w$ and $u$ are both in the thick part and that all subsurface projections
between $\mu_u$ and $\mu_w$ are uniformly bounded. 
\corref{Bounded-Distance} implies that $d_\calT(u,w)=O(1)$. 
\end{proof}

\bibliographystyle{alpha}
\bibliography{../main}

\begin{thebibliography}{MMS10}

\bibitem[Ber78]{bers:EP}
L.~Bers.
\newblock An extremal problem for quasiconformal mappings and a theorem by
  {T}hurston.
\newblock {\em Acta Math.}, 141(1-2):73--98, 1978.

\bibitem[CR07]{rafi:TL}
Y.~Choi and K.~Rafi.
\newblock Comparison between {T}eichm\"uller and {L}ipschitz metrics.
\newblock {\em J. Lond. Math. Soc. (2)}, 76(3):739--756, 2007.

\bibitem[CRS08]{rafi:LT}
Y.~Choi, K.~Rafi, and C.~Series.
\newblock Lines of minima and {T}eichm\"uller geodesics.
\newblock {\em Geom. Funct. Anal.}, 18(3):698--754, 2008.

\bibitem[GL00]{gardiner:QT}
F.~P. Gardiner and N.~Lakic.
\newblock {\em Quasiconformal {T}eichm\"uller theory}, volume~76 of {\em
  Mathematical Surveys and Monographs}.
\newblock American Mathematical Society, Providence, RI, 2000.

\bibitem[Hub06]{hubbard:TT}
J.~Hubbard.
\newblock {\em {T}eichm\"uller theory and applications to geometry, topology
  and dynamics}.
\newblock Matric Edition, Ithaca, NY, 2006.

\bibitem[Ker80]{kerckhoff:AG}
S.~P. Kerckhoff.
\newblock The asymptotic geometry of {T}eichm\"uller space.
\newblock {\em Topology}, 19(1):23--41, 1980.

\bibitem[Kla99]{klarreich:BC}
E.~Klarreich.
\newblock The boundary at infinity of the curve complex and the relative
  {T}eichm\"uller space.
\newblock Preprint, 1999.

\bibitem[LR10]{rafi:LQC}
A.~Lenzhen and K.~Rafi.
\newblock Length of a curve is quasi-convex along a {T}eichm\"uller g3eodesic.
\newblock Preprint, 2010.

\bibitem[Mas82]{masur:IE}
Howard Masur.
\newblock Interval exchange transformations and measured foliations.
\newblock {\em Ann. of Math. (2)}, 115(1):169--200, 1982.

\bibitem[Min96]{minsky:PR}
Y.~N. Minsky.
\newblock Extremal length estimates and product regions in {T}eichm\"uller
  space.
\newblock {\em Duke Math. J.}, 83(2):249--286, 1996.

\bibitem[MM99]{minsky:CCI}
H.~A. Masur and Y.~N. Minsky.
\newblock Geometry of the complex of curves. {I}. {H}yperbolicity.
\newblock {\em Invent. Math.}, 138(1):103--149, 1999.

\bibitem[MM00]{minsky:CCII}
H.~Masur and Y.~N. Minsky.
\newblock Geometry of the complex of curves. {II}. {H}ierarchical structure.
\newblock {\em Geom. Funct. Anal.}, 10(4):902--974, 2000.

\bibitem[MMS10]{masur:TS}
H.~Masur, L~Mosher, and S.~Schleimer.
\newblock On train track splitting sequences.
\newblock arXiv:1004.4564, 2010.

\bibitem[MW95]{masur:NH}
H.~A. Masur and M.~Wolf.
\newblock {T}eichm\"uller space is not {G}romov hyperbolic.
\newblock {\em Ann. Acad. Sci. Fenn. Ser. A I Math.}, 20(2):259--267, 1995.

\bibitem[Raf05]{rafi:SC}
K.~Rafi.
\newblock A characterization of short curves of a {T}eichm\"uller geodesic.
\newblock {\em Geometry and Topology}, 9:179--202, 2005.

\bibitem[Raf07a]{rafi:CM}
K.~Rafi.
\newblock {A combinatorial model for the {T}eichm\"uller metric}.
\newblock {\em Geom. Funct. Anal.}, 17(3):936--959, 2007.

\bibitem[Raf07b]{rafi:TT}
K.~Rafi.
\newblock Thick-thin decomposition for quadratic differentials.
\newblock {\em Math. Res. Lett.}, 14(2):333--341, 2007.

\bibitem[RS09]{rafi:CC}
K.~Rafi and S.~Schleimer.
\newblock Covers and the curve complex.
\newblock {\em Geom. Topol.}, 13(4):2141--2162, 2009.

\bibitem[Thu88]{thurston:GD}
W.~P. Thurston.
\newblock On the geometry and dynamics of diffeomorphisms of surfaces.
\newblock {\em Bull. Amer. Math. Soc. (N.S.)}, 19(2):417--431, 1988.

\bibitem[Thu98]{thurston:MSM}
W.~P. Thurston.
\newblock {Minimal stretch maps between hyperbolic surfaces}.
\newblock Preprint, 1998.

\bibitem[Vee86]{veech:TGF}
W.~A. Veech.
\newblock The {T}eichm\"uller geodesic flow.
\newblock {\em Ann. of Math. (2)}, 124(3):441--530, 1986.

\end{thebibliography}

\end{document}